\documentclass[12pt]{article} 
\usepackage{setspace}
\usepackage[utf8]{inputenc} 

\usepackage[margin=1in]{geometry}
\usepackage{graphicx} 
\usepackage{epstopdf}

\usepackage{booktabs} 
\usepackage{array} 
\usepackage{paralist} 
\usepackage{verbatim} 

\usepackage{fancyhdr} 
\usepackage{sectsty}
\allsectionsfont{\sffamily\mdseries\upshape} 

\usepackage[nottoc,notlof,notlot]{tocbibind} 
\usepackage[titles]{tocloft} 

\usepackage{amsmath, amsthm, amssymb}

\usepackage{multirow}
\usepackage{algorithm}
\usepackage{algorithmic}

\usepackage{xcolor}
\usepackage{booktabs}
\usepackage{comment}
\usepackage{amsmath}
\usepackage{mathrsfs}
\usepackage{float}
\usepackage{epstopdf}
\usepackage{multirow}
\usepackage{rotating}
\usepackage{amssymb}

\usepackage{subcaption}

\usepackage{amsthm}
\newtheorem{thm}{Theorem}
\newtheorem{lem}{Lemma}
\newtheorem{prop}{Proposition}
\newtheorem{cor}{Corollary}

\newtheorem{dfn}{Definition}

\DeclareMathOperator{\Tr}{Tr}
\usepackage{rotating}
\usepackage{pdflscape}

\usepackage{tikz,pgfplots}
\usetikzlibrary{matrix}
\usepgfplotslibrary{groupplots}
\pgfplotsset{compat=newest}
\usetikzlibrary{calc,positioning,shapes,fit,arrows,shadows}
\tikzstyle{line} = [ draw, -latex']  
\usetikzlibrary{shapes,arrows}
\usepgfplotslibrary{units}
\usepackage{url}
\usepackage{color}

\graphicspath{{figures/}} 

\usepackage{color}

\title{On Subadditive Duality \\ for Conic Mixed-Integer Programs} 
\author{Burak Kocuk\thanks{burakkocuk@sabanciuniv.edu, Industrial Engineering Program,  Sabanc{\i} University, Istanbul, Turkey 34956.} \and
Diego A. Mor{\'{a}}n R.\thanks{diego.moran@uai.cl, School of Business,  Universidad Adolfo Ib{\'{a}}\~nez, Santiago, Chile 7941169.}} 

\date{} 

\usepackage{color}
\def\blue{\color{blue}}
\def\black{\color{black}}

\def \LL{ {\mathcal{L}}}

\def \FF{ {\mathcal{F}}}

\def \rr{ {\mathbb{R}}}
\def \zz{ {\mathbb{Z}}}

\def \conv{\textup{conv}}

\def \int{ {\textup{int}}}
\def \rec{\textup{rec.cone}}

\def \tq{{\,:\,}}
\def \aff{{\textup{aff}}}

\def \MM{\mathbb{M}}

\def \RR{\mathbb{R}}

\def\blue{\color{black}}
\def\black{\color{black}}

\def\damr{\color{black}}

\begin{document}
\maketitle

\begin{abstract}
In this paper, we show that the subadditive dual of a feasible conic mixed-integer program (MIP) is a strong dual whenever it is feasible. Moreover, we show that this dual feasibility condition is equivalent to feasibility of the conic dual of the continuous relaxation of the conic MIP. In addition, we prove that all known conditions and other `natural' conditions for strong duality, such as strict mixed-integer feasibility, boundedness of the feasible set or essentially strict feasibility imply that the subadditive dual is feasible. As an intermediate result, we extend the so-called `finiteness property' from full-dimensional convex sets to  intersections of  full-dimensional convex sets and Dirichlet convex sets.
\end{abstract}

\section{Introduction}
\label{section:SectionLabel}

\

{\bf Duality for mixed-integer programs (MIPs):} Duality is an important concept in mathematical programming for both analyzing the properties of optimization problems and constructing solution methods.
For an optimization problem in minimization (resp., maximization) form, a dual problem is called \textit{weak} if its optimal value provides a lower (resp., upper) bound to the optimal value of the primal problem. Furthermore, a dual problem is called  \textit{strong} if: ({\blue i}) There is {\em zero duality gap}, that is, its optimal value is equal to that of the primal problem, and ({\blue ii}) It is {\em solvable}, that is, the optimal value is attained. The first property ensures that the dual problem is giving the best possible bound and the second property provides a way to obtain this best bound by finding an optimal dual feasible solution. Both of these properties are crucial in the development of effective optimization algorithms.

It is well-known that linear programming (LP) and conic programming (CP) problems and their respective duals satisfy strong duality under mild conditions (such as boundedness and feasibility or strict feasibility) \cite{BenTal01}. The case of mixed-integer linear programs (MILP) is more involved and requires the definition of a functional dual problem, the so-called subadditive dual, which is a strong dual when the data defining the primal problem is rational \cite{Guzelsoy, Nemhauser}. Recently, this latter duality result was extended to conic MIPs  in \cite{Moran} under a mixed-integer strict feasibility requirement, similar to the one needed in the continuous conic case. We also note that other types of duals have been studied in the case of general mixed-integer nonlinear programming (MINLP) problems. For instance, the Karush-Kuhn-Tucker (KKT) optimality conditions are generalized for MINLPs in \cite{Baes}, Lagrangian-based methods are used in \cite{chen2010, feizollahi2017exact}, and other geometric \cite{basu2016optimalit} or algebraic \cite{santana2017some} approaches are utilized to obtain strong duality results in particular cases.

{\bf Conic MIPs:}  Conic MIP problems generalize MILPs and have significantly more expressive power in terms of modeling. To name just a few application areas, conic MIPs are used in options pricing  \cite{Pinar13}, 
power distribution systems \cite{Kocuk2015}, 
Euclidean $k$-center problems \cite{Brandenberg}  and engineering  design  \cite{Dai}. We note here that all the conic MIPs used in these applications include {\black binary variables, and that this feature, rather than being the exception, is a general rule when modeling real life problems.}

In spite of the growing interest in conic MIP applications, conic MIP solvers are not as mature as their MILP counterparts. Although the subadditive dual for linear/conic MIPs do not yield straightforward solution procedures, any dual feasible solution generates a valid inequality for the primal problem. Moreover, if strong duality holds, all cutting planes are equal to or dominated by a cutting plane obtained from such a solution \cite{wolsey1981integer, Moran}. 
We know that these valid inequalities are extremely useful for MILPs (see, for instance, \cite{bixby2007progress}) and one may expect that cutting planes designed for conic MIPs will help solve them more efficiently. 
Recent work on cutting planes for conic MIPs includes generalizations of Gomory cuts \cite{cezik1},  
rounding cuts \cite{Atamturk10,Sanjeevi2016}, 
split/disjunctive cuts \cite{Yildiz} 
and minimal valid inequalities  \cite{Fatma2016}. 

{\bf Our contributions:} In this paper, we study sufficient conditions under which the subadditive dual for conic MIPs is a strong dual. {\damr Although the mixed-integer strict feasibility condition in~\cite{Moran} is somewhat reasonable, it has two main drawbacks}: ({\blue i}) It may not always hold in practical problem settings (for instance, {\black if the conic MIP includes binary variables or its feasible region is not full-dimensional}), and ({\blue ii}) It is not `easy' to check in general, {\blue that is, there is no known polynomial time algorithm for this task}. {\damr This motivates the search for other sufficient conditions.}

Our main result is the following sufficient condition for strong duality: {\em If the primal conic MIP and the subadditive dual problems are both feasible, then strong duality holds}. {\damr Furthermore, under the assumption that the primal problem is feasible, we show that feasibility of the subadditive dual is equivalent to feasibility of the conic dual of the continuous relaxation of the primal.}

Moreover, we prove that under some `natural' conditions, which include all known sufficient conditions for strong duality in the {\damr linear and conic MIP cases}, the conic dual of the continuous relaxation of the conic MIP is feasible. As a consequence of our main result, we obtain that these {\black `natural'} conditions, some of which include cases that are `easy' to check, imply that the subadditive dual is a strong dual.

Finally, as an intermediate result {\black of independent interest}, we find a sufficient condition for the finiteness of the objective function of the {\damr convex} MIP implying the finiteness of the objective function of its continuous relaxation. This is an extension of the `finiteness property' result given  {\blue in \cite{Meyer1974} for rational polyhedra and
 in~\cite{Moran} for full-dimensional convex sets to 
   intersections of full-dimensional convex sets and Dirichlet convex sets} (a class of sets first studied in~\cite{DM2011}).

{\bf Organization of the rest of the paper:} In  Section~\ref{sec:main}, we  review some results from the literature and {\damr give precise statements of our main contributions}.
The proof of the new sufficient condition for strong duality is presented in Section~\ref{sec:proof1}. 
In Section~\ref{sec:proof2}, we study properties of Dirichlet convex sets and give a proof of our extension of the finiteness property. 
In Section~\ref{sec:proof3}, {\damr we prove that some generally occurring conditions on the primal conic MIP imply dual feasibility}.  
Finally, some concluding remarks are discussed in Section~\ref{section:conc}.

\section{Main results}
\label{sec:main}

\subsection{Preliminaries}

For a set $X\subseteq\rr^n$, we denote its interior as $\int(X)$, its recession cone as $\rec(X)$, its affine hull as $\aff(X)$, its (affine) dimension as $\dim(X)$ and its convex hull as $\conv(X)$. We next review some definitions that are necessary to {\damr formulate} a conic MIP.

\begin{dfn}[Regular cone]
	A cone $K \subseteq \RR^m$ is called regular if it is closed, convex, pointed and full-dimensional.
\end{dfn} 


\begin{dfn}[Conic inequality]
	A conic inequality with respect to a regular cone $K$ is defined as $x \succeq_{K} y$, meaning that  $x - y \in K$. 
	We will denote the relation $x \in \int(K)$ alternatively as  $x \succ_K 0$.
\end{dfn}

We define a generic conic MIP as follows:

\begin{equation}\label{eq:generic}
\begin{aligned}
z^* := \inf  &\hspace{0.5em}  c^Tx + d^Ty\\
\mathrm{s.t.}   &\hspace{0.5em} A x + Gy \succeq_K b \\
& \hspace{0.5em}  x  \in   \mathbb{Z}^{n_1}, y \in   \mathbb{R}^{n_2},
\end{aligned}
\end{equation}
where $K \subseteq \RR^m$ is a regular cone, {\black $c \in \RR^{n_1}$, $d \in \RR^{n_2}$}, $A \in \RR^{m \times n_1}$, $G \in \RR^{m \times n_2}$ and $b \in \RR^{m}$.

The following definitions are instrumental in the {\damr description} of the subadditive dual problem of~\eqref{eq:generic}.

\begin{dfn}[Subadditive function]
	A function $f:\RR^m \to \RR$ is subadditive if $f(u+v) \le f(u) + f(v)$ for all $u,v \in \RR^m$.
\end{dfn}

\begin{dfn}[Non-decreasing function]
	A function $f:\RR^m \to \RR$ is non-decreasing with respect to a regular cone $K \subseteq \RR^m$  if $u \succeq_K v \Rightarrow f(u) \ge f(v)$.
\end{dfn}

We denote the set of subadditive functions that are non-decreasing with respect to a regular cone $K \subseteq \RR^m$ as $\mathcal{F}_K$ and for $f\in\mathcal{F}_K$ we denote $\bar f(x) := \lim\sup_{\delta\to0^+}\frac{f(\delta x )}{\delta}$.

The subadditive dual problem of~\eqref{eq:generic} is
\begin{equation}\label{eq:dualgeneric}
\begin{aligned}
\rho^*  :=  \sup  &\hspace{0.5em}  f(b) \\
\mathrm{s.t.}   &\hspace{0.5em} f(A^j) = -f(-A^j) = c_j &j=1,\dots,n_1 \\
&\hspace{0.5em} \bar f(G^j) = -\bar f(-G^j) = d_j &j=1,\dots,n_2 \\
& \hspace{0.5em}  f(0) = 0 \\
& \hspace{0.5em}  f \in \mathcal{F}_K,
\end{aligned}
\end{equation}
where $A^j$ (resp. $G^j$) denote the $j^\text{th}$ column of the matrix $A$ (resp. $G$).

It is not hard to show that the subadditive dual~\eqref{eq:dualgeneric} is a weak dual to the primal conic MIP~\eqref{eq:generic}, that is, any dual feasible solution provides a lower bound for the optimal value of the primal (see, for instance, Proposition 3.2 in \cite{Moran}). The following result provides a sufficient condition for the subadditive dual to be a strong dual for~\eqref{eq:generic}, that is, there is zero duality gap {\damr(i.e., $z^*=\rho^*$)} and the subadditive dual is solvable {\damr(i.e., there exists a function $f$ feasible for the dual  such that $f(b)=\rho^*$)}.

\begin{thm}(Theorem 2.4 in  \cite{Moran}) \label{Theorem} 
 	If $z^* > -\infty$ and there exists $(\hat x, \hat y) \in \mathbb{Z}^{n_1}    \times \mathbb{R}^{n_2}$ such that  $A\hat x + G \hat y \succ_K b$, then the dual problem~\eqref{eq:dualgeneric} is a strong dual for~\eqref{eq:generic}.
\end{thm}

We note here that the sufficient condition in Theorem~\ref{Theorem}   is similar to the strict feasibility condition for strong duality in the conic case (see Condition {\em (b.)} in Theorem~\ref{sdual_cp} below with $A_2=0$, $b_2=0$).

\begin{thm}[Duality for conic programming \cite{shapiro2003duality, BenTal01}]\label{sdual_cp} Let $A_1\in \rr^{m_1 \times n}$, $A_2\in \rr^{m_2 \times n}$, $c \in \rr^n$, $b_1 \in \rr^{m_1}$, $b_2 \in \rr^{m_2}$ and let $K\subseteq\rr^{\blue m_1}$ be a regular cone. Consider the primal conic program $\inf\{c^Tx \tq A_1x \succeq_K b_1,A_2x\geq b_2\}$ and its corresponding dual conic program $\sup\{b_1^T \lambda_1 +b_2^T \lambda_2 \tq A_1^T \lambda_1 +A_2^T \lambda_2 = c,\ \lambda_1 \succeq_{K_*}0, \lambda_2 \geq0\}$, where $K_* :=  \{\lambda\in \mathbb{R}^m: \lambda^T x \ge 0, \ \forall x \in K\}$ is the dual cone to   $K$. Then,
\begin{enumerate}
	\item Weak duality always holds.
	\item If either (a.) {\blue the} feasible region of the continuous relaxation is bounded, or (b.) {\blue there} exists an essentially strictly feasible point, that is, a point $\hat{x}\in\rr^n$ such that $A_1\hat{x} \succ_K b_1$ and $A_2 \hat x\geq b_2$, then strong duality holds.
\end{enumerate}
\end{thm}


\subsection{A new sufficient condition for strong duality}

Although the strict feasibility conditions in Theorem~\ref{Theorem} and Theorem~\ref{sdual_cp} { are somewhat analogous, in the MIP case the condition has a crucial limitation: it is not satisfied for conic MIPs that include binary variables and/or have equality constraints in its formulation, or in general if the conic set does not include a mixed-integer point in its interior (see Figure~\ref{Theo1_picture} below). These are very important cases, as conic MIPs with these characteristics often arise when modeling real-life problems.

\begin{figure}[H]
\centering
\begin{subfigure}{6.25cm}
\centering
\begin{tikzpicture}
\coordinate (A00) at (0,0);
\coordinate (A01) at (0,1);
\coordinate (A02) at (0,2);
\coordinate (A03) at (0,3);
\coordinate (A10) at (1,0);
\coordinate (A11) at (1,1);
\coordinate (A12) at (1,2);
\coordinate (A13) at (1,3);
\coordinate (A20) at (2,0);
\coordinate (A21) at (2,1);
\coordinate (A22) at (2,2);
\coordinate (A23) at (2,3);
\coordinate (A30) at (3,0);
\coordinate (A31) at (3,1);
\coordinate (A32) at (3,2);
\coordinate (A33) at (3,3);
\coordinate (A105) at (1,0.25);
\coordinate (A305) at (3,0.5);
\coordinate (A2305) at (2.3,0.5);
\coordinate (A232) at (2.3,2);
\coordinate (A316) at (3,1.6);
\filldraw[gray, fill opacity=0.3] (A12) -- (A232) -- (A316) -- (A305) -- (A105) -- (A12);
\draw[fill=black] (A00) circle (0.05) node[above right] { };%
\draw[fill=black] (A01)  circle (0.05) node[above right] { };
\draw[fill=black] (A02)  circle (0.05) node[above right] { };
\draw[fill=black] (A03)  circle (0.05) node[above right] { };
\draw[fill=black] (A10)  circle (0.05) node[above right] { };
\draw[fill=red] (A11) circle (0.05) node[above right] { };
\draw[fill=red]  (A12)  circle (0.05) node[above right] { };
\draw[fill=black]  (A13) circle (0.05) node[above right] { };
\draw[fill=black] (A20) circle (0.05) node[above right] { };
\draw[fill=red] (A21)  circle (0.05) node[above] {$\hat x$};
\draw[fill=red] (A22)  circle (0.05) node[above right] { };
\draw[fill=black] (A23)  circle (0.05) node[above right] { };
\draw[fill=black] (A30) circle (0.05) node[above right] { };
\draw[fill=black] (A31) circle (0.05) node[above right] { };
\draw[fill=black] (A32) circle (0.05) node[above right] { };
\draw[fill=black] (A33) circle (0.05) node[above right] { };
\filldraw[gray, fill opacity=0.3] (1.15,1.85) circle (1.42cm);
%
\end{tikzpicture}
\caption{There exist a strictly feasible mixed-integer point.}
\end{subfigure}
\begin{subfigure}{6.25cm}
\centering
\begin{tikzpicture}
\coordinate (A00) at (0,0);
\coordinate (A01) at (0,1);
\coordinate (A02) at (0,2);
\coordinate (A03) at (0,3);
\coordinate (A10) at (1,0);
\coordinate (A11) at (1,1);
\coordinate (A12) at (1,2);
\coordinate (A13) at (1,3);
\coordinate (A20) at (2,0);
\coordinate (A21) at (2,1);
\coordinate (A22) at (2,2);
\coordinate (A23) at (2,3);
\coordinate (A30) at (3,0);
\coordinate (A31) at (3,1);
\coordinate (A32) at (3,2);
\coordinate (A33) at (3,3);
\coordinate (A105) at (1,0.25);
\coordinate (A305) at (3,0.5);
\coordinate (A2305) at (2.3,0.5);
\coordinate (A232) at (2.3,2);
\coordinate (A316) at (3,1.6);
\filldraw[gray, fill opacity=0.3] (A12) -- (A232) -- (A316) -- (A305) -- (A105) -- (A12);
\draw[fill=black] (A00) circle (0.05) node[above right] { };%
\draw[fill=black] (A01)  circle (0.05) node[above right] { };
\draw[fill=black] (A02)  circle (0.05) node[above right] { };
\draw[fill=black] (A03)  circle (0.05) node[above right] { };
\draw[fill=black] (A10)  circle (0.05) node[above right] { };
\draw[fill=red] (A11) circle (0.05) node[above right] { };
\draw[fill=red]  (A12)  circle (0.05) node[above right] { };
\draw[fill=black]  (A13) circle (0.05) node[above right] { };
\draw[fill=black] (A20) circle (0.05) node[above right] { };
\draw[fill=black] (A21)  circle (0.05) node[above right] { };
\draw[fill=red] (A22)  circle (0.05) node[above right] { };
\draw[fill=black] (A23)  circle (0.05) node[above right] { };
\draw[fill=black] (A30) circle (0.05) node[above right] { };
\draw[fill=black] (A31) circle (0.05) node[above right] { };
\draw[fill=black] (A32) circle (0.05) node[above right] { };
\draw[fill=black] (A33) circle (0.05) node[above right] { };
\filldraw[gray, fill opacity=0.3] (1.0,2) circle (1.25cm);
%
\end{tikzpicture}
\caption{There is no strictly feasible mixed-integer point.}
\end{subfigure}
\caption{Two cases comparing the applicability of Theorem~\ref{Theorem}.}
\label{Theo1_picture}
\end{figure}
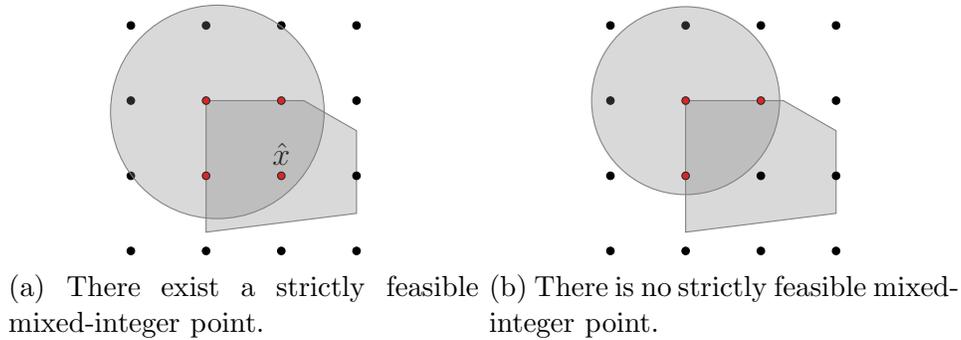

In this paper, we generalize all the known sufficient conditions for strong duality by showing that whenever the subadditive dual is feasible, it is a strong dual. Notice that when considering {\damr the subadditive dual problem~\eqref{eq:dualgeneric}}, this sufficient condition is the most general possible. {\damr Furthermore, we show that this condition is equivalent to feasibility of the conic dual of the continuous relaxation of the conic MIP.} We state our result in the theorem below.

\begin{thm}\label{Theorem:Bounded} Assume that the primal problem~\eqref{eq:generic} is feasible, then:
\begin{enumerate}
\item If the subadditive dual problem~\eqref{eq:dualgeneric} is feasible, then it is a strong dual for~\eqref{eq:generic}.
\item {\damr The subadditive dual~\eqref{eq:dualgeneric} is feasible if and only if the {conic} dual of the continuous relaxation of~\eqref{eq:generic} is feasible.}
\end{enumerate}
\end{thm}

We obtain the following corollary.

\begin{cor}\label{Cor:Bounded}
	If the primal problem~\eqref{eq:generic}  and the conic dual of its continuous relaxation are both feasible, then~\eqref{eq:dualgeneric} is a strong dual for~\eqref{eq:generic}.
\end{cor}

{For a feasible conic MIP}, notice that the condition in Corollary~\ref{Cor:Bounded} is simpler than the one given in Theorem~\ref{Theorem}, as it only requires checking the feasibility of a conic set rather than finding a mixed-integer point in the interior of a  conic set, which is an NP-complete problem even for polyhedral sets {\black\cite[Chapter~18.1]{schrijver1998theory}}. 
Furthermore, in this paper we show that some `natural' conditions for~\eqref{eq:dualgeneric} to be a strong dual for~\eqref{eq:generic}, including the sufficient conditions known in the literature, actually imply that the conic dual of the continuous relaxation of~\eqref{eq:generic} is feasible (see Section~\ref{section:result3}).

\subsection{Finiteness property for  Dirichlet convex sets}

Next we study the finiteness property, a relationship between the finiteness of the optimal value of a general convex MIP and  the finiteness of the optimal value of its continuous relaxation. 

{\blue
\begin{dfn}[Finiteness property]
A convex set $X {\black \subseteq \mathbb{R}^{n_1+n_2}}$ is said to have the finiteness property with respect to $\zz^{n_1}\times\rr^{n_2}$ if for every vector $c \in  \mathbb{R}^{n_1+n_2} $, we have
$\sup \{  c^Tx \tq x \in X \}<+\infty$ if and only if $\sup \{ c^Tx  \tq  x \in X \cap ( \zz^{n_1}\times\rr^{n_2} ) \}<+\infty $.
\end{dfn}
}

The finiteness property is closely related to strong duality. A sufficient condition for general convex sets to have this property is stated in \cite{Moran}. In this paper, we show a more general sufficient condition for the finiteness property to hold {\damr and then use it to} show the results {\damr that we describe} in Section~\ref{section:result3} {\damr below}. Before stating {\black the new sufficient condition}, we need some definitions.

\begin{dfn}[Mixed-lattices]
A mixed-lattice is a set of the form
$$\MM=\{Ax+By\tq x\in\zz^{n_1},\ y\in\rr^{n_2}\},$$
where  $A\in\rr^{m\times {n_1}}$, $B\in \rr^{m\times {n_2}}$ and the set $\LL=\{Ax\tq x\in\zz^{n_1}\}$ is a lattice contained in $V^\perp$, where $V=\{By\tq y\in\rr^{n_1}\}$. If $\MM=\{Ax\tq x\in\zz^{n_1}\}$, that is, there are no `continuous' components, we simply call $\MM$ a lattice.

\end{dfn} 

The following definition is a slight modification of the definition of ``Convex sets with the Dirichlet property'' that appears in \cite{DM2011}.

\begin{dfn}[Dirichlet convex sets]
	A convex set $P\subseteq \rr^n$ is said to be a Dirichlet convex set with respect to a mixed-lattice $\MM$ if for all $z\in P\cap \MM$, $r\in \rec(P)$ and for all $\epsilon>0$, $\gamma\geq0$ there exists a point $w\in P\cap \MM$ at a {\black (Euclidean)} distance less than $\epsilon$ from the half-line $\{z + \lambda r\tq \lambda \geq \gamma\}$. 
\end{dfn} 

Examples of Dirichlet convex sets with respect to $\MM=\zz^{n_1}\times\rr^{n_2}$ are, for instance, bounded sets,  rational polyhedra and strictly convex sets (see Proposition~\ref{Dsetsprop} and Corollary~\ref{corProp2} in Section~\ref{sec:proof2}).

The finiteness property {\blue was first proven for rational polyhedra in \cite{Meyer1974}. A related} result in \cite{Moran} states that if a convex set $X$ contains a mixed-integer point in its interior, then $X$ has the finiteness property (see Proposition 4.5 in \cite{Moran}). We extend {\blue the latter result} as follows. 

\begin{thm}\label{Dfiniteness}
Let $\MM\subseteq \rr^n$ be a mixed-lattice, let $c\in\rr^n$ and let $P\subseteq \rr^n$ be a Dirichlet convex set with respect to $\MM$. Let $S=P\cap \MM$ and let $X\subseteq\rr^n$ be a closed convex set such that $\int(X)\cap S\neq\emptyset$. Then,
$$\sup_{x\in X \cap S} c^Tx<+\infty\qquad\Leftrightarrow\qquad\sup_{x\in X\cap P} c^Tx<+\infty.$$
\end{thm}

{\blue Theorem~\ref{Dfiniteness} extends the result in \cite{Moran} from the case $\MM=\zz^{n_1}\times \rr^{n_1}$ and $P=\rr^n$ to arbitrary mixed-integer lattices $\MM$ and Dirichlet convex sets $P$.}

\subsection{Some natural sufficient conditions that imply dual feasibility}\label{section:result3}

Consider the following conic MIP

\begin{equation}\label{eq:ESF}
\begin{aligned}
{\black z^{*}} := \inf  &\hspace{0.5em}  c^Tx + d^Ty\\
\mathrm{s.t.}   &\hspace{0.5em} A_1 x + G_1y \succeq_{K_1} b_1 \\
&\hspace{0.5em} A_2x + G_2y \succeq_{K_2} b_2 \\
& \hspace{0.5em}  x  \in   \mathbb{Z}^{n_1}, y \in   \mathbb{R}^{n_2},
\end{aligned}
\end{equation}
where  $c \in \RR^{n_1}$,  $d \in \RR^{n_2}$,  and {\blue for $i=1,2$}, $K_i \subseteq \RR^{m_i}$ is a regular cone,   $A_i \in \RR^{m_i \times n_1}$, $G_i \in \RR^{m_i \times n_2}$, $b_i \in \RR^{m_i}$.

{\damr Let us} denote $S_i=\{(x,y)\in \rr^{n_1}\times\rr^{n_2}\tq A_ix + G_iy \succeq_{K_i} b_i\}$, for $i=1,2$.
As a consequence of {\blue Theorem~\ref{sdual_cp}} 
and Theorem~\ref{Dfiniteness}, we obtain the following result.

\begin{thm}\label{TheoremESF}
Suppose that the optimal value of~\eqref{eq:ESF} is finite, $\int(S_1)\cap S_2\cap (\zz^{n_1}  \times \rr^{n_2})\neq \emptyset$ and one of the following conditions is satisfied: 
\begin{enumerate}
	\item[{\bf A.}] The set $S_2$ is bounded.
	\item[{\bf B.}] The set $S_2$ is a rational polyhedron.
\end{enumerate}
Then, the {conic} dual of the continuous relaxation of~\eqref{eq:ESF} is feasible. 
\end{thm}

Note that by Corollary~\ref{Cor:Bounded}, the assumptions of Theorem~\ref{TheoremESF} also imply that the subadditive dual of~\eqref{eq:ESF} is a strong dual.

The assumptions in Theorem~\ref{TheoremESF} include the following special cases of `natural' sufficient conditions for strong duality of the subadditive dual: 
\begin{enumerate}
\item 
The conic MIP has  a bounded feasible region.
\item 
The conic MIP is a linear MIP defined by rational data.
\item
The conic MIP satisfies the mixed-integer strict feasibility condition in Theorem~\ref{Theorem} \cite{Moran}.
\item
The conic MIP set is mixed-integer essentially  strictly feasible, that is, there exists a mixed-integer feasible point  that satisfies the nonlinear conic constraints strictly. 
\item
The conic MIP includes binary variables and either satisfies (i) or (iv). 

\end{enumerate}

Indeed, when {\bf A.} holds we obtain {\damr Condition} (i) and when {\bf B.} holds we obtain {\damr Condition} (iv). Conditions (ii) and (iii) can be seen as the special cases of  {\damr Condition} (iv) that we obtain when $S_1=\rr^{n_1}\times\rr^{n_2}$  and $S_2=\rr^{n_1}\times\rr^{n_2}$, respectively. Finally,  {\damr Condition} (v) is a special case of  {\damr Condition} (i) or  {\damr Condition} (iv).

We note here that Conditions (ii) and (iii) were already known in the literature and that, to the best of our knowledge, Conditions (i), (iv) and (v) are new. 

\section{Proof of Theorem~\ref{Theorem:Bounded} and some examples}
\label{sec:proof1}


In this section, we first show an extension of Theorem~\ref{Theorem} (see Proposition~\ref{Theorem:Binary}) and use this result to give a proof for Theorem~\ref{Theorem:Bounded}. Then, we illustrate the use of our results by {\black giving} two examples.


{\blue
\subsection{Strong duality for conic MIPs with binary variables}
\label{sec:binary prob}

In this section, we consider the following  conic MIP,
\begin{equation}\label{eq:binary prob}
	\begin{aligned}
	z^* = \inf  &\hspace{0.5em}  c^Tx + d^Ty  + h^T w \\
	\mathrm{s.t.}   &\hspace{0.5em} A  x + G y + H w \succeq_K b    \\
	& \hspace{0.5em}  x  \in   \mathbb{Z}^{n_1}  ,  y \in   \mathbb{R}^{n_2} , w \in \{0,1\}^{n_3} ,
	\end{aligned}
\end{equation}
where $K \subseteq \RR^m$ is a regular cone,  $c \in \RR^{n_1}$, $d \in \RR^{n_2}$, $h \in \RR^{n_3}$, $A \in \RR^{m \times n_1}$, $G \in \RR^{m \times n_2}$, $H \in \RR^{m \times n_3}$ and $b \in \RR^{m}$.

We show that  the subadditive dual of \eqref{eq:binary prob} is a strong dual when we require  $z^*>-\infty$ and strict feasibility of the conic constraints. 

\begin{prop} \label{Theorem:Binary}
If $z^*>-\infty$ and there exists $(\hat x, \hat y, \hat w) \in   \mathbb{Z}^{n_1}  \times \mathbb{R}^{n_2} \times  \{0,1\}^{n_3} $ such that  $A\hat x + G \hat y + H \hat w \succ_K b$, then the subadditive dual of~\eqref{eq:binary prob}, {\black that is,}
\begin{equation}\label{eq:binary dualprob}
\begin{aligned}
 \rho^* =  \sup  &\hspace{0.5em} F(b, 0, -e) \\
  \mathrm{s.t.}   &\hspace{0.5em} F(A^j,0,0) = -F(-A^j, 0,0)  = c_j &j=1,\dots,n_1  \\
&\hspace{0.5em} \bar F(G^j, 0,0) = -\bar F(-G^j,0,0)  = d_j &j=1,\dots,n_2  \\
 &\hspace{0.5em} F(H^j, e^j,-e^j) = -F(-H^j, -e^j,e^j)  = h_j &j=1,\dots,n_3  \\
  & \hspace{0.5em}  F(0,0,0) = 0 \\
  & \hspace{0.5em}  F \in \mathcal{F}_{K\times \mathbb{R}_+^{n_3}\times \mathbb{R}_+^{n_3}},
\end{aligned}
\end{equation}
is a strong dual.  {\blue Here, vectors $e$ and $e^j$ respectively denote the vector of ones and the $j^{\text{th}}$ unit vector in $\mathbb{R}^{n_3}$.}

\end{prop}
\begin{proof}
First of all, by construction,~\eqref{eq:binary dualprob} is a weak dual for~\eqref{eq:binary prob} (see \cite{Moran}). Hence, we have $\rho^* \le z^*$.

 Let $\epsilon \in (0,1)$ and consider the ``perturbed" primal problem:
\begin{equation}\label{eq:binary pertprob}
\begin{aligned} 
 z' = \inf   &\hspace{0.5em}  c^Tx + d^Ty + h^T w\\
  \mathrm{s.t.}   &\hspace{0.5em} A  x + G y + H w \succeq_K b    \\
  & \hspace{0.5em}   w  \ge  {-\epsilon e}   \\
  & \hspace{-0.56em}   -w  \ge  {-(1+\epsilon) e}   \\
	& \hspace{0.5em}  x  \in   \mathbb{Z}^{n_1}  ,  y \in   \mathbb{R}^{n_2} , w \in \mathbb{Z}^{n_3} .
\end{aligned}
\end{equation}

Due to the choice of $\epsilon$, the feasible regions of~\eqref{eq:binary prob} and~\eqref{eq:binary pertprob} coincide. Thus, we obtain that $z'=z^*$. We also note that point   $(\hat x, \hat y, \hat w) \in   \mathbb{Z}^{n_1}  \times \mathbb{R}^{n_2} \times  \{0,1\}^{n_3} $ satisfies all the constraints of~\eqref{eq:binary pertprob} strictly. Therefore, we can apply Theorem~\ref{Theorem} to~\eqref{eq:binary pertprob} and its subadditive dual
\begin{equation}\label{eq:binary dualpertprob}
\begin{aligned}
 \rho' =  \sup  &\hspace{0.5em} f(b,  {-\epsilon e},  {-(1+\epsilon)e}) \\
  \mathrm{s.t.}   &\hspace{0.5em} f(A^j, 0,0) = -f(-A^j, 0,0)  = c_j   &j=1,\dots,n_1 \\
&\hspace{0.5em} \bar f(G^j, 0,0) = -\bar f(-G^j,0,0)  = d_j   &j=1,\dots,n_2 \\
  &\hspace{0.5em} f(H^j, e^j,-e^j) = -f(-H^j, -e^j,e^j)  = h_j   &j=1,\dots,n_3 \\
  & \hspace{0.5em}  f(0,0,0) = 0 \\
  & \hspace{0.5em}  f \in \mathcal{F}_{K\times \mathbb{R}_+^{n_3}\times \mathbb{R}_+^{n_3}}.
\end{aligned}
\end{equation}
Since $z'=z^*>-\infty$, we conclude that there exists a function $f':\rr^m\times\rr^{n_3}\times\rr^{n_3}\rightarrow\rr$ that is an optimal solution to~\eqref{eq:binary dualpertprob} and such that $f'(b,  {-\epsilon e},  {-(1+\epsilon)e})=\rho'=z'$. Since the feasible regions of~\eqref{eq:binary dualprob} and~\eqref{eq:binary dualpertprob} are the same, the function defined as $f'$ is also a feasible solution to~\eqref{eq:binary dualprob}. Also, we have 
\[
\rho^* \ge f'(b,0,-e) \ge f'(b,  {-\epsilon e},  {-(1+\epsilon)e}) = \rho'=z'=z^*,
\]
where the first inequality follows since $f'$ is a feasible solution to~\eqref{eq:binary dualprob}, and the last inequality follows due to the fact that $f'$ is a non-decreasing function with respect to $K\times \mathbb{R}_+^{n_3}\times \mathbb{R}_+^{n_3}$. Combining the previous facts, we conclude that $z^* = \rho^*$ and~\eqref{eq:binary dualprob} is solvable, proving that it is a strong dual.
\end{proof}
}

\subsection{Strong duality for conic MIPs with feasible subadditive dual}
\label{sec:bounded}


Before giving the proof of Theorem~\ref{Theorem:Bounded} we introduce some notation. For any $u\in \rr^m$ we define
\begin{equation*}
\vartheta^*_{MIP}(u) :=  \inf \big \{ c^Tx + d^Ty \tq A  x  + G  y  \succeq_K  { u }, x  \in   \mathbb{Z}^{n_1}  ,  y \in   \mathbb{R}^{n_2} \big \}.
\end{equation*}
The function $\vartheta^*_{MIP}:\rr^m\rightarrow \rr \cup\{-\infty,+\infty\}$ is the value function of~\eqref{eq:generic} and, in particular, $\vartheta^*_{MIP}(b)=z^*$, $\vartheta^*_{MIP}(u)=-\infty$ when the objective function of the conic MIP is not bounded below and $\vartheta^*_{MIP}(u)=+\infty$ when the conic MIP is infeasible.

\begin{proof}[Proof of Theorem~\ref{Theorem:Bounded}] \mbox{}
\begin{enumerate} 
	\item {First of all}, by construction,~\eqref{eq:dualgeneric} is a weak dual for~\eqref{eq:generic} (\cite{Moran}). Hence, we have $\rho^* \le z^*$. Also, observe that $z^*>-\infty$, as the dual~\eqref{eq:dualgeneric} is feasible and  any dual feasible solution provides a lower bound for the optimal value of the primal~\eqref{eq:generic}.

Now, let $(\hat x, \hat y)\in\zz^{n_1}\times\rr^{n_2}$ be a feasible solution to the primal~\eqref{eq:generic} and let  $v \in K$ such that $A  \hat x  + G  \hat y + v \succ_K  { b} $ (for instance, any $v\succ_K0$ satisfies this condition). Since $v\in K$, we have that any feasible point for the primal~\eqref{eq:generic} is also feasible for the primal with right-hand side $b-v$ (as $b\succeq_{K} b-v$), and thus $\vartheta^*_{MIP}(b-v)\leq \vartheta^*_{MIP}(b)=z^*$. On the other hand, since the subadditive dual~\eqref{eq:dualgeneric} is feasible, we obtain that the subadditive dual  for the primal problem with right-hand side $b-v$ is also feasible (as feasibility of the subadditive dual does not depend on the right-hand side of the primal), and hence, by weak duality we have $\vartheta^*_{MIP}(b-v)>-\infty$. 

 Consider the ``perturbed" primal problem
\begin{equation} \label{eq:Mpertgeneric}
\begin{aligned} 
 z' = \inf   &\hspace{0.5em}  c^Tx + d^Ty + [\vartheta^*_{MIP}(b)-\vartheta^*_{MIP}(b-v)] w\\
  \mathrm{s.t.}   &\hspace{0.5em} A  x + G y + v w\succeq_K b  \\
  & \hspace{0.5em}  x  \in   \mathbb{Z}^{n_1}  ,  y \in   \mathbb{R}^{n_2}, w \in \{0,1\}.
\end{aligned}
\end{equation}
We have that $z'=\min\{\vartheta^*_{MIP}(b)+0,\vartheta^*_{MIP}(b-v)+(\vartheta^*_{MIP}(b)-\vartheta^*_{MIP}(b-v))\}=z^*$. 
We also note that the vector  $(\hat x, \hat y, 1)\in\zz^{n_1}\times\rr^{n_2}\times \zz$ satisfies  the conic constraint strictly, hence, we can apply Proposition~\ref{Theorem:Binary} to~\eqref{eq:Mpertgeneric} and its subadditive dual
\begin{equation} \label{eq:dualMpertgeneric}
\begin{aligned}
\rho' = \sup  &\hspace{0.25em} f({b}, {0}, {-1}) \\
  \mathrm{s.t.}   &\hspace{0.25em} f(A^j, {0}, {0}) = -f(-A^j, {0}, {0})  = c_j  &j=1,\dots,n_1   \\
&\hspace{0.25em} \bar f(G^j, {0}, {0}) = -\bar f(-G^j, {0}, {0})  = d_j  &j=1,\dots,n_2\\
&\hspace{0.25em}  { f(v,0,-1) = -f(-v,0,1)  = \vartheta^*_{MIP}(b)-\vartheta^*_{MIP}(b-v) } \,  \\
  & \hspace{0.25em}  f(0, {0}, {0}) = 0 \\
  & \hspace{0.25em}  f \in \mathcal{F}_{K\times  {\mathbb{R}_+}\times {\mathbb{R}_+}}.
\end{aligned}
\end{equation}
Since $z'=z^*>-\infty$, we conclude that there exists a function $f':\rr^m\times\rr\times\rr\rightarrow\rr$ that is an optimal solution to~\eqref{eq:dualMpertgeneric} and such that $f({b}, {0}, {-1})=\rho'=z'$. It is easy to show that  the function $F:\rr^m\rightarrow\rr$ defined as $F(u):=f'(u,0,0)$ is a feasible solution to~\eqref{eq:dualgeneric}. Furthermore, we have 
\[
\rho^* \ge F(b) = f'(b,0,0) \ge f'(b,  0,-1) = \rho'=z'=z^*,
\]
where the first inequality follows since $F$ is a feasible solution to~\eqref{eq:dualgeneric}, and the last inequality follows due to the fact that $f'$ is a non-decreasing function with respect to $K\times  {\mathbb{R}_+}\times {\mathbb{R}_+}$. Combining the previous facts, we conclude that $z^* = \rho^*$ and~\eqref{eq:dualgeneric} is solvable, proving that it is a strong dual.

{\damr \item  
For any $u\in \rr^m$, we define the conic program  
\begin{equation*}
CP(u): \   \inf \big \{ c^Tx + d^Ty \tq A  x  + G  y  \succeq_K  { u }, x  \in   \mathbb{R}^{n_1}  ,  y \in   \mathbb{R}^{n_2}  \},
\end{equation*}
and denote  its value function as $\vartheta_{CP}^*(u)$. Note that $CP(b)$ corresponds to the continuous relaxation of~\eqref{eq:generic}.

We start by showing that feasibility of the conic dual of $CP(b)$ implies that the subadditive dual is feasible. Indeed, for any $\lambda$ feasible for the conic dual of  $CP(b)$, it is easy to check that the function $f:\rr^m\rightarrow\rr$ defined by $f(d)=\lambda^Td$ is feasible for~\eqref{eq:dualgeneric}.

Now we show that feasibility of the subadditive dual~\eqref{eq:dualgeneric} implies that the conic dual of $CP(b)$ is feasible. 
Let $v \succ_K 0$.
Since the subadditive dual~\eqref{eq:dualgeneric} is feasible {(independent of the right-hand side of the primal~\eqref{eq:generic})}, we have that 
$ \vartheta^*_{MIP}(b-v) > -\infty $. Note that  $CP(b-v)$ is strictly feasible 
since the conic MIP \eqref{eq:generic} is feasible. Therefore, by the finiteness property (Proposition 4.5 in \cite{Moran}), we have that $\vartheta_{CP}^*(b-v) > -\infty$ since 
$ \vartheta^*_{MIP}(b-v) > -\infty $. Hence, we conclude that $CP(b-v)$ is both strictly feasible and bounded below, implying that its conic dual is feasible due to Theorem~\ref{sdual_cp}. Finally, since the feasible regions of the conic dual of $CP(b-v)$ and $CP(b)$ coincide, we conclude that the conic dual of $CP(b)$ 
 is feasible.
}
\end{enumerate}
\end{proof}

 We note here that although the assumptions in Proposition~\ref{Theorem:Binary} can be shown to be a special case of the assumption in Theorem~\ref{Theorem:Bounded} (see Section~\ref{section:result3}), we decided to prove Proposition~\ref{Theorem:Binary} separately since on one hand, it simplifies the proof of Theorem~\ref{Theorem:Bounded} and, on the other hand, the proof technique is slightly different: in the proof of Proposition~\ref{Theorem:Binary} we perturb the conic MIP~\eqref{eq:binary prob} only by changing the right-hand side vector whereas in the proof of Theorem~\ref{Theorem:Bounded} we perturb the primal~\eqref{eq:generic}  by introducing an auxiliary variable with an appropriate objective function coefficient. This difference is also noticeable on how we retrieve the optimal dual function of the conic MIP ~\eqref{eq:binary prob} (resp.,~\eqref{eq:dualgeneric}) from the optimal solution of the subadditive dual of the perturbed conic MIP~\eqref{eq:binary dualpertprob} (resp.,~\eqref{eq:dualMpertgeneric}): the optimal dual solution to~\eqref{eq:binary prob} is exactly the same optimal solution to~\eqref{eq:binary dualpertprob}, and the optimal dual solution to~\eqref{eq:dualgeneric} is a restriction of the solution to~\eqref{eq:dualMpertgeneric}.

\subsection{Examples}
The following examples, which  are adapted from \cite{BenTal01}, feature  two  feasible, below bounded conic MIPs. In the  first example, the subadditive dual is infeasible (hence, the duality gap is infinite). In the second example, {\black the} conic MIP is not strictly mixed-integer feasible (therefore, Theorem~\ref{Theorem} is not applicable) but its subadditive dual is a strong dual since it is feasible.

\noindent
{\bf Example 1.} Let $L^3 := \{x\in\mathbb{R}^3: \ \sqrt{x_1^2+x_2^2} \le x_3\}$ denote the 3-dimensional Lorentz cone.
Consider the following conic MIP:
\begin{equation}\label{eq:conicMIPex1}
\begin{aligned}
\inf_{x  \in   \mathbb{Z}^{2}}  &\hspace{0.5em}  x_2\\
\mathrm{s.t.}   &\hspace{0.5em} \begin{bmatrix} 1 & 0 \\ 0 & 1 \\ 1 & 0 \end{bmatrix} \begin{bmatrix} x_1 \\ x_2 \end{bmatrix} \succeq_{L^3} \begin{bmatrix} 0 \\ 0 \\ 0 \end{bmatrix},
\end{aligned}
\end{equation}
Observe that the primal problem~\eqref{eq:conicMIPex1} is below bounded (since we have $x_2=0$ in any feasible solution) but not strictly feasible. We claim that its subadditive dual   is infeasible. To prove our claim, we consider the following perturbed primal problem
\begin{equation}\label{eq:conicMIPex1p}
\begin{aligned}
\inf_{x  \in   \mathbb{Z}^{2}}  &\hspace{0.5em}  x_2\\
\mathrm{s.t.}   &\hspace{0.5em} \begin{bmatrix} 1 & 0 \\ 0 & 1 \\ 1 & 0 \end{bmatrix} \begin{bmatrix} x_1 \\ x_2 \end{bmatrix} \succeq_{L^3} \begin{bmatrix} 0 \\ 0 \\ -\epsilon \end{bmatrix},
\end{aligned}
\end{equation}
for $\epsilon > 0$. Notice that the  perturbed primal problem~\eqref{eq:conicMIPex1p} is unbounded below as any $(x_1, x_2)$ with $x_2 \in \mathbb{Z}_-$ and $x_1 = \left \lceil \frac{x_2^2 - \epsilon^2}{2 \epsilon} \right \rceil $ is feasible. Therefore, the subadditive dual of~\eqref{eq:conicMIPex1p} is infeasible, which implies that the subadditive dual of  problem~\eqref{eq:conicMIPex1} is also infeasible as these two subadditive duals share the same feasible region.

\noindent
{\bf Example 2.} Let $\mathbb{S}_+^3 := \{X\in\mathbb{R}^{3\times3}: \ u^T X u \ge 0 \ \forall u\in \mathbb{R}^3\}$ denote the cone of  $3\times3$ positive semidefinite matrices. 
Consider the following conic MIP:
\begin{equation}\label{eq:conicMIPex2}
\begin{aligned}
\inf_{x  \in   \mathbb{Z}^{2}}  &\hspace{0.5em}  x_2\\
\mathrm{s.t.}   &\hspace{0.5em} \begin{bmatrix} 0 & 0 & 0 \\ 0 & 1 & 0 \\ 0 & 0 & 0 \end{bmatrix}x_1 +  \begin{bmatrix} 1 & 0 & 0 \\ 0 & 0 & 1 \\ 0 & 1 & 0 \end{bmatrix}x_2 \succeq_{\mathbb{S}_+^3} \begin{bmatrix} -1 & 0 & 0 \\ 0 & 0 &  0 \\ 0 & 0 & 0 \end{bmatrix},
\end{aligned}
\end{equation}
and its subadditive dual:
\begin{equation}  \label{eq:sadDualex2}
\begin{aligned}
 \sup_f  &\hspace{0.25em} f \left ( \begin{bmatrix} -1 & 0 & 0 \\ 0 & 0 &  0 \\ 0 & 0 & 0 \end{bmatrix} \right) \\
  \mathrm{s.t.}   &\hspace{0.25em}  f \left ( \begin{bmatrix} 0 & 0 & 0 \\ 0 & 1 & 0 \\ 0 & 0 & 0 \end{bmatrix}  \right)  =  -f \left ( -\begin{bmatrix} 0 & 0 & 0 \\ 0 & 1 & 0 \\ 0 & 0 & 0 \end{bmatrix}  \right)  =  0    \\
  &\hspace{0.25em}  f \left ( \begin{bmatrix} 1 & 0 & 0 \\ 0 & 0 & 1 \\ 0 & 1 & 0 \end{bmatrix}  \right)  =  -f \left ( -\begin{bmatrix} 1 & 0 & 0 \\ 0 & 0 & 1 \\ 0 & 1 & 0 \end{bmatrix}\right)  =  1    \\
&\hspace{0.25em}  f \left ( \begin{bmatrix} 0 & 0 & 0 \\ 0 & 0 & 0 \\ 0 & 0 & 0 \end{bmatrix}  \right)  = 0 \\
  & \hspace{0.25em}  f \in \mathcal{F}_{\mathbb{S}_+^3}.
\end{aligned}
\end{equation}
Notice that the primal problem~\eqref{eq:conicMIPex2} is below bounded (since we have $x_2=0$ in any feasible solution)  but not strictly feasible. 
We observe that $\Lambda = e^1 ({e^1})^T$ is a feasible solution for the conic dual of the continuous relaxation of~\eqref{eq:conicMIPex2}, {\damr where $e^1 = \begin{bmatrix} 1 & 0 & 0 \end{bmatrix}^T$}. Therefore,
the function $f:\rr^{3\times3}\rightarrow\rr $ defined as
$
 f(A) = \Tr(\Lambda A) = A_{11}
 $ 
is a feasible solution for the  subadditive dual~\eqref{eq:sadDualex2} implying that
it is a strong dual.

\section{Properties of Dirichlet convex sets and the proof of Theorem~\ref{Dfiniteness}}
\label{sec:proof2}



%

%

In this section we first study some properties of Dirichlet convex set that allow us to show that some important classes of convex sets such as bounded convex sets, {\black strictly convex sets} and rational polyhedra are Dirichlet convex sets, and then give a proof of Theorem~\ref{Dfiniteness}.

\subsection{Dirichlet convex sets}
\label{sec:DCS}

%

We start by stating a lemma on intersection of mixed-lattices and linear subspaces.

\begin{lem}[Lemma 3.13 in \cite{DM2016}]\label{MinterW}
	Let $\MM\subseteq \rr^n$ be a mixed-lattice and let $W\subseteq \rr^n$ be a linear subspace. Then $\MM\cap W$ is a mixed-lattice.	
\end{lem}

In order to establish a sufficient condition for a convex set to be a Dirichlet convex set, we need the following lemma, which is a consequence of the Dirichlet's Diophantine approximation theorem and appears in this form for the case $\MM=\zz^n$ in~\cite{BCCZ2010}. 

\begin{lem}[Basu et al.~\cite{BCCZ2010}]\label{dlemma}
	Let $\MM\subseteq \rr^n$ be a mixed-lattice.  If $z \in \MM$ and $r \in \aff(\MM)$, then for all $\epsilon >0$ and $\gamma \geq 0$, there exists a point of $w\in \MM$ at a distance less than $\epsilon$ from the half-line $\{z + \lambda r\tq \lambda \geq \gamma\}$.
\end{lem}

Although  the extension to general mixed-lattices is a straightforward consequence of the result in~\cite{BCCZ2010}, we still present the proof of Lemma~\ref{dlemma} in Appendix~\ref{app:proofs},  {\black for completeness.}


The following result is a special case of a classic result in convex analysis.  
{\black We present a proof in Appendix~\ref{app:proofs} that is adapted} from the proof of Theorem 18.1 in \cite{rockafellar1970}.

\begin{lem}\label{Fintegerhull}
Let $\MM\subseteq \rr^n$ be a mixed-lattice and let  $X\subseteq \rr^n$ be a convex set. Then for any face $F$ of $X$, we have $\conv(F\cap\MM)=\conv(X\cap\MM)\cap F$.
\end{lem}

The next lemma establishes that the property of being a Dirichlet convex set is invariant under invertible affine mappings that preserve the corresponding mixed-lattice.

\begin{lem}\label{forgottenlemma}
Let $\MM\subseteq \rr^n$ be a mixed-lattice and let  $X\subseteq \rr^n$ be a Dirichlet convex set. Let $T :\rr^n\rightarrow \rr^n $ be an invertible affine mapping such that $T(\MM)=\MM$. Then $T(X)$ is a Dirichlet convex set with respect to $\MM$.
	
\end{lem}

\begin{proof} Let $T(x)=Ax+b$, where $A$ is an invertible $n\times n$ matrix and $b\in\rr^n$.
	
		First, observe that since $X$ is a convex set and $T$ is an affine mapping, then $T(X)=AX+b$ is indeed a convex set. 
		
		We now show that $T(X)$ is a Dirichlet convex set. Let $y\in T(X)\cap \MM$, $s\in \rec(T(X))$, $\epsilon>0$ and $\gamma\geq0$. We must show that there exists a point $v\in T(X)\cap\MM$ at a distance less than $\epsilon$ from the half-line $\{y + \lambda s\tq \lambda \geq \gamma\}$.
		
		Since $T$ is invertible and $T(\MM)=\MM$, there exists $z\in X\cap \MM$ such that $T(z)=y$. Moreover, due to the fact that $T$ is an affine mapping, there exists $r\in \rec(X)$ such that $Ar=s$.
		
		Since $X$ is a Dirichlet convex set with respect to $\MM$, we have that there exists a point $w\in X\cap\MM$ at a distance less than $\epsilon/\|A\|_2$ from the half-line $\{z + \lambda r\tq \lambda \geq \gamma\}$, that is, $\|w-z-\lambda^*r \|_2\leq\epsilon/\|A\|_2$ for some $\lambda^*\geq\gamma$. Let $v=T(w)$ and notice that $v\in T(X)\cap\MM$ as $w\in X\cap\MM$. We obtain that
		\begin{align*}
		\|v-y-\lambda^*s\|_2&=\|T(w)-T(z)-\lambda^*Ar\|_2\\
		&=\|Aw+b -Az-b-\lambda^*Ar\|_2\\
		&\leq \|A\|_2\|w-z-\lambda^*r\|_2\\
		&\leq \epsilon.
		\end{align*}
		
		This implies that the distance between $v$ and the half-line $\{y + \lambda s\tq \lambda \geq \gamma\}$ is less than $\epsilon$. Therefore, we conclude that $T(X)$ is a Dirichlet convex set with respect to~$\MM$.
\end{proof}
We remark that not all   operations preserve the  property of being a Dirichlet convex set. For instance, the intersection of two Dirichlet convex sets is not necessarily a Dirichlet convex set, {\black see Example 2.9 in \cite{DM2016} for an illustration of this fact}.

The most important result in this section is a sufficient condition for a convex set to be a Dirichlet convex set, that we state in the next proposition.

\begin{prop}\label{Dsetsprop}
	Let $\MM\subseteq \rr^n$ be a mixed-lattice and let  $X\subseteq \rr^n$ be a convex set such that  $\rec(X)=\rec(\conv(X\cap\MM))$. Then $X$ is a Dirichlet convex set with respect to $\MM$.
\end{prop}
\begin{proof}
	Let $z\in X\cap \MM$. Then by Lemma~\ref{forgottenlemma} applied to $\MM$, $X$ and $T(x)=x-z$, we conclude that the set $X$ is a Dirichlet convex set with respect to $\MM$ if and only if  $X-z$ is a Dirichlet convex set with respect to $\MM$. Therefore, we may assume for the rest of the proof that the set $X$ contains the origin, and thus $W=\aff(X\cap\MM)$ is a linear subspace (we will use this latter fact in order to apply Lemma~\ref{MinterW} in {\bf Case 1.} below). 
	
	Let	$z\in X\cap \MM$, $r\in \rec(X)$, $\epsilon>0$ and $\gamma\geq0$. We must show that there exists a point $w\in X\cap\MM$ at a distance less than $\epsilon$ from the half-line $\{z + \lambda r\tq \lambda \geq \gamma\}$.
	
	We will use induction {\damr on} the dimension of $X$. Clearly, if $\dim(X)=0$, the result is true. Now, we assume that  any convex set  $X'\subseteq \rr^n$ with $\dim(X')<\dim(X)$  such that  $\rec(X')=\rec(\conv(X'\cap\MM))$ is a Dirichlet convex set with respect to $\MM$. We have two cases.
	
	\begin{itemize}
		\item {\bf Case 1: the half-line $\{z + \lambda r\tq \lambda \geq \gamma'\}$ is contained in the relative interior of $X$ for some $\gamma'\geq \gamma$.}
		
		Since $\{z + \lambda r\tq \lambda \geq \gamma'\}$ is contained in the relative interior of $X$, there exists $\epsilon'>0$  such that $\epsilon'< \epsilon$ and any point in $\aff(X)$ at distance less than $\epsilon'$ from the half-line $\{z + \lambda r\tq \lambda \geq \gamma'\}$ belongs to $X$.
		
		Recall the linear subspace $W=\aff(X\cap\MM)$ and let $\MM'=\MM\cap W$. By Lemma~\ref{MinterW}, we obtain that $\MM'\subseteq W$ is a mixed-lattice. Moreover, by definition of $W$ and $\MM'$, we obtain $W=\aff(\MM')$ and since $r\in \rec(X)=\rec(\conv(X\cap\MM))$ and $X\cap \MM \subseteq W $ we obtain that $r\in \rec(W)=\aff(\MM')$. On the other hand, since $z\in X\cap \MM$, we obtain $z\in \MM'$. We can apply Lemma~\ref{dlemma} and conclude that there exists a point $w\in \MM'$ at a distance less than $\epsilon'$ from the half-line $\{z + \lambda r\tq \lambda \geq \gamma'\}$. Since $w\in\MM'$, we obtain that $w\in W\subseteq \aff(X)$, and therefore, by the selection of $\epsilon'$ and since $\gamma'\geq\gamma$, we conclude that $w\in X\cap\MM$ and that $w$ is at distance less than $\epsilon$ from the half-line $\{z + \lambda r\tq \lambda \geq \gamma\}$.  
		
		\item {\bf Case 2:  the half-line $\{z + \lambda r\tq \lambda \geq \gamma\}$ is contained in a proper face $F$ of $X$.}
		
		Since $F$ is a proper face of $X$, we have $\dim(F)<\dim(X)$. Moreover, by Lemma~\ref{Fintegerhull}, we obtain that $\conv(F\cap\MM)=\conv(X\cap\MM)\cap F$. Furthermore, since $\conv(F\cap\MM)\neq\emptyset$, we obtain
		\begin{align*}
		\rec(\conv(F\cap\MM))&=\rec(\conv(X\cap\MM))\cap \rec(F)\\
		&=\rec(X)\cap \rec(F)\\
		&=\rec(F).
		\end{align*}
		Therefore, we can apply the induction hypothesis to $F$ and conclude that $F$ is a Dirichlet convex set with respect to $\MM$. Since the half-line $\{z + \lambda r\tq \lambda \geq \gamma\}$ is contained in $F$, we obtain that $z\in F\cap \MM$ and that $r\in \rec(F)$. Therefore, we conclude that there exists a point  $w\in F\cap\MM\subseteq X\cap \MM$ at a distance less than $\epsilon$ from the half-line $\{z + \lambda r\tq \lambda \geq \gamma\}$.
	\end{itemize}
\end{proof}

When the mixed-integer lattice is $\MM=\zz^{n_1}\times\rr^{n_2}$, some examples of convex sets $X$ satisfying the assumption $\rec(X)=\rec(\conv(X\cap\MM))$ in Proposition~\ref{Dsetsprop} are: bounded convex sets, rational polyhedra (\cite{Meyer1974}), closed strictly convex sets (\cite{DM2013}, \cite{DM2016}) and  closed convex sets whose recession cone is generated by vectors in $\MM$ (see Corollary 1 in \cite{DM2013} for a proof of this assertion in the case $\MM=\zz^n$ and the recession cone of the convex set being a rational polyhedral cone). 
Based on the discussion above, we obtain the following corollary of Proposition~\ref{Dsetsprop}.
\begin{cor}\label{corProp2} The following are Dirichlet convex sets with respect to $\zz^{n_1}\times\rr^{n_2}$: bounded convex sets, rational polyhedra, closed strictly convex {\damr sets} and  closed convex sets whose recession cone is generated by vectors in $\zz^{n_1}\times\rr^{n_2}$.	
\end{cor}

If a convex set $X$ does not contain lines, then Proposition  \ref{Dsetsprop} can be seen as an extension of Proposition 4.7 in \cite{DM2011}. Indeed, if $\rec(X)=\rec(\conv(X\cap\MM))$ holds, then $\conv(X\cap\MM)$ is closed (see Theorem 3.20 in \cite{DM2016}).
Conversely, if a closed convex set $X$ satisfies that $\conv(X\cap\MM)$ is a closed set and contains a mixed-lattice vector in its relative interior, then it can be shown that  $\rec(X)=\rec(\conv(X\cap\MM))$ (see Theorem 2 in \cite{DM2013} for the proof of this result when $\MM=\zz^n$ and the convex set is full-dimensional).

\subsection{The finiteness property}

Let $S\subseteq \rr^n$. A full-dimensional convex set $Q\subseteq\rr^n$ is said to be an {\em $S$-free convex set} if $\int(Q)\cap S=\emptyset$. $Q$ is said to be a {\em maximal $S$-free convex set} if it is not strictly contained in any other $S$-free convex set. {\damr When $S= P\cap\MM$, where $P$ is a convex set and $\MM$ is a mixed-lattice, Averkov \cite{Averkov2013} showed that maximal $S$-free convex sets are polyhedra (see Theorem 2.4 in~\cite{Averkov2013}).}

The following lemma gives a property of maximal $S$-free sets in the case $S$ is defined as the mixed-lattice points contained in a Dirichlet convex set. {\damr This lemma is crucial in the proof of Theorem~\ref{Dfiniteness}}.

\begin{lem}\label{dprop}
	Let $\MM\subseteq \rr^n$ be a mixed-lattice and $P\subseteq \rr^n$ be a Dirichlet convex set with respect to $\MM$. Let $S=P\cap \MM$, and let $Q$ be a full-dimensional maximal $S$-free convex set.  Then if $r\in \rec(P\cap Q)$, then {\damr$-r\in \rec(Q)$}.
\end{lem}

\begin{proof}
	Let  {\damr $Q'=\{x-\lambda r\tq x\in Q,\lambda\geq0 \}$. In order to prove the lemma,  we will show that $Q'=Q$. Since $Q'$ is a convex set, $Q\subseteq Q'$} and $Q$ is a maximal $S$-free convex set, it suffices to show that $Q'$ is an $S$-free convex set. Assume for a contradiction that $\int(Q')\cap S\neq \emptyset$. Then there exists $x\in \int(Q)$ and $\gamma\geq0$ such that $z=x-\gamma r\in S$. Since $x\in\int(Q)$ and $r\in\rec(Q)$, there exists $\epsilon>0$ such that 
	{\damr the set $\mathcal{H}_\varepsilon:=\{x + \lambda r\tq \lambda \geq 0\}+B(0,\epsilon)$}
	 is contained in $\int(Q)$. On the other hand, since $z\in S$ and $P$ is a Dirichlet convex set with respect to $\MM$, we have that there exists a point  $z'\in S$ at a distance less than $\epsilon$ from the half-line
	 {\damr $\{z + \lambda r\tq \lambda \geq \gamma\}$. Since $\{z + \lambda r\tq \lambda \geq \gamma\}=\{x + \lambda r\tq \lambda \geq 0\}$, we obtain  that $z'\in \mathcal{H}_\varepsilon\subseteq \int(Q)$}, a contradiction with the fact that $Q$ is an $S$-free set. {\damr Therefore, we conclude that $Q'=Q$, as desired.}
\end{proof}

Now we are ready to prove Theorem~\ref{Dfiniteness}.


\begin{proof}[Proof of Theorem~\ref{Dfiniteness}] 
We only need to show {\damr $\sup \{  c^Tx \tq x \in X\cap S \}<+\infty$ implies $\sup \{ c^Tx  \tq  x \in X \cap P \}<+\infty $}.

Let ${\damr z^*}=\sup\{c^Tx \tq x\in X \cap S\} $ and let $X^\geq=\{x\in X \tq c^Tx\geq {\damr z^*} \}$. Assume for a contradiction that $\sup\{c^Tx \tq x\in X\cap P\} =+\infty$. Then $X^\geq$ is a full-dimensional unbounded $S$-free convex set. Therefore, there exists a full-dimensional maximal $S$-free polyhedron $Q\supseteq X^\geq$ {\damr (by Theorem 2.4 in~\cite{Averkov2013})}.

On the other hand, since $X$ is not $S$-free, we obtain $X\nsubseteq Q$, so there exists $x_0\in X\setminus Q$. In particular, as $Q$ is a polyhedron, there exists an inequality $a^Tx\leq b$ of $Q$ such that $a^Tx_0>b$. Also notice that since $x_0\notin Q$, we have $x_0\notin X^\geq$ and thus $c^Tx_0<{\damr z^*}$.

Let $\{x_k\}_{k\geq 1}\subseteq X^\geq\cap P$ be such that $c^Tx\rightarrow+\infty$ and for each $k\geq 1$ define 
$$y_k= \left[\frac{c^Tx_k-{\damr z^*}}{c^Tx_k-c^Tx_0}\right]x_0+ \left[\frac{{\damr z^*}-c^Tx_0}{c^Tx_k-c^Tx_0}\right]x_k,$$
{\damr this sequence is well-defined since we have $c^Tx_k\geq z^* > c^Tx_0$ for all $k\geq 1$}. 

Notice that for all $k\geq 1$
$$c^Ty_k=\frac{c^Tx_0(c^Tx_k-{\damr z^*})+c^Tx_k({\damr z^*}-c^Tx_0)}{c^Tx_k-c^Tx_0}={\damr z^*},$$
and that $y_k \in X$ as it is a convex combination of $x_0\in X$ and $x_k\in X^\geq\subseteq X$, and $X$ is a convex set.

Now as $X^\geq\cap P\subseteq P\cap Q$, we obtain $\rec(X^\geq\cap P)\subseteq \rec(P\cap Q)$. {\damr Let $L=\aff(\rec(X^\geq\cap P))$ and notice that $L$ is a linear subspace.} Then, by Lemma~\ref{dprop}, we obtain:
$$(X^\geq\cap P) + L\subseteq Q.$$

Moreover, by the definition of the linear subspace $L$, we have that $(X^\geq\cap P)\cap L^\perp$ is a bounded set. Therefore, for any $f\in\rr^n$, $\sup\{f^Tx\tq x\in (X^\geq\cap P) + L\}=+\infty$ if and only if there exists $l\in L$ such that $f^Tl>0$ if and only if $\inf\{f^Tx\tq x\in (X^\geq\cap P) + L\}=-\infty$. Since $a^Tx\leq b$ is a valid inequality for $(X^\geq\cap P)$ (as $(X^\geq\cap P) + L\subseteq Q$), the latter properties imply that $a^Tx\geq b-\eta$ is valid for $(X^\geq\cap P) + L$ for some $\eta>0$.

Observe that:
\begin{align*}
a^Ty_k&=\frac{a^Tx_0(c^Tx_k-{\damr z^*})+a^Tx_k({\damr z^*}-c^Tx_0)}{c^Tx_k-c^Tx_0}\\
&> \frac{b(c^Tx_k-{\damr z^*})+(b-\eta)({\damr z^*}-c^Tx_0)}{c^Tx_k-c^Tx_0}\\
&\geq b + \frac{-\eta {\damr z^*}+\eta c^Tx_0}{c^Tx_k-c^Tx_0},
\end{align*}
where the first inequality follows from: (i) $a^Tx_0>b$ and $c^Tx_k\geq {\damr z^*}$, and (ii) $a^Tx_k\geq b-\eta$ (since $\{x_k\}_{k\geq 1}\subseteq X^\geq\cap P$ and $a^Tx\geq b-\eta$ is valid for $(X^\geq\cap P) + L$) and ${\damr z^*}>c^Tx_0$ (since $x_0\notin Q$, and thus $x\notin X^\geq$).

Therefore, as $\frac{-\eta {\damr z^*} +\eta c^Tx_0}{c^Tx_k-c^Tx_0}\rightarrow0$ as $k\rightarrow+\infty$, for large enough $\bar K \ge 1$, we have $a^Ty_{\bar K} >b$. On the other hand, since $y_{\bar K}\in X$ and $c^Ty_{\bar K}={\damr z^*}$ we obtain $y_{\bar K}\in X^\geq$. Thus, since $X^\geq\subseteq Q $, $y_{\bar K}\in Q$ and therefore $a^Ty_{\bar K}\leq b$, a contradiction.
\end{proof}

\section{Proof of Theorem~\ref{TheoremESF}}
\label{sec:proof3}

The proof of Theorem~\ref{TheoremESF} is a consequence of  {\damr Theorem~\ref{sdual_cp}} 
 and Theorem~\ref{Dfiniteness}.

\begin{proof}[Proof of Theorem~\ref{TheoremESF}]

We first recall  Conditions {\bf A.} and {\bf B.} in Theorem~\ref{TheoremESF}: 

\begin{enumerate}
	\item[{\bf A.}] The set $S_2$ is bounded.
	\item[{\bf B.}] The set $S_2$ is a rational polyhedron.
\end{enumerate}	
	
We will show that if the optimal value of~\eqref{eq:ESF} is finite, $\int(S_1)\cap S_2\cap (\zz^{n_1}  \times \rr^{n_2})\neq \emptyset$ and one of the assumptions above is satisfied, then the conic dual of the continuous relaxation of~\eqref{eq:ESF}  is feasible.	
	
Observe that under either Assumption {\bf A.} or Assumption {\bf B.} the set  $S_2$ is a Dirichlet convex  set with respect to the mixed-lattice $\zz^{n_1}  \times \rr^{n_2}$ (see Corollary~\ref{corProp2}).  Therefore, we can use Theorem~\ref{Dfiniteness} with $X=S_1$, $P=S_2$ and $\MM=\zz^{n_1}  \times \rr^{n_2}$ to conclude that the optimal value of the continuous relaxation of~\eqref{eq:ESF} is finite. 

Now, since the nonlinear conic constraints in~\eqref{eq:ESF} are strictly feasible (that is $\int(S_1)\cap S_2\neq \emptyset$) and either {\bf A.} or {\bf B.} is satisfied, then we have that either Condition~{\em (a.)} or Condition~{\em (b.)} in Theorem~\ref{sdual_cp} holds. Therefore, the continuous relaxation of~\eqref{eq:ESF} and its {\damr conic} dual satisfy strong duality. Moreover, since the optimal value of the continuous relaxation of~\eqref{eq:ESF} is finite, we conclude that the conic dual of the continuous relaxation of~\eqref{eq:ESF} is solvable, and therefore it must be feasible.
\end{proof}

\section{Final remarks}
\label{section:conc}

By weak duality, any function $F \in\FF_K$ (that is, $F$ subadditive and non-decreasing with respect to the cone $K$)  that satisfies $F(0)=0$ is a {\em cut-generating function} (see, for instance,~\cite{CGF2014}) and, in particular, defines the following valid inequality for the feasible region of the conic MIP~\eqref{eq:generic}:
\begin{equation}\label{sad_ineq}
\sum_{j=1}^{n_1}F(A^j)x_j+\sum_{j=1}^{n_2}\bar F( {G^j} )y_j\geq F(b).
\end{equation}
Conversely, the strong duality result in Theorem~\ref{Theorem:Bounded} yields the following corollary.
\begin{cor}
	Let $\pi^Tx+\gamma^Ty\geq \pi_o$ be a valid inequality for the feasible region of the conic MIP~\eqref{eq:generic} and suppose that there exists a  subadditive function $f \in\FF_K$ satisfying $f(0)=0$, $f(A^j) = -f(-A^j) = \pi_j, j=1,\dots,n_1$ and $\bar f(G^j) = -\bar f(-G^j) = \gamma_j, j=1,\dots,n_2$. Then, there exists a subadditive function $F \in\FF_K$ such that $F(0)=0$, $F(A^j) = -F(-A^j) = \pi_j, j=1,\dots,n_1$, $\bar F(G^j) = -\bar F(-G^j) = \gamma_j, j=1,\dots,n_2$ and $F(b)\geq \pi_o$. In particular, the valid inequality of the form~\eqref{sad_ineq} defined by $F$ dominates $\pi^Tx+\gamma^Ty\geq \pi_o$. 
\end{cor}

A similar result was proven in \cite{Moran} (see Corollary 6.1). We emphasize here that we could use the value function of the conic MIP~\eqref{eq:generic}, $\vartheta^*_{MIP}:\rr^m\rightarrow \rr \cup\{-\infty,+\infty\}$ to generate a valid inequality of the form~\eqref{sad_ineq} (see Proposition 4.8 in \cite{Moran}). Furthermore, by appropriately changing the objective function in~\eqref{eq:generic}, {\black and considering the associated value function,} we could generate all valid inequalities for~\eqref{eq:generic} in this way. However, one disadvantage of this approach as compared to {\damr using} functions in $\FF_K$ is
{\black that value functions of conic MIPs are difficult to compute. Moreover, as value functions are in general not finite-valued everywhere, they cannot be  cut-generating functions (by definition).}

As a final remark, we can combine the proofs of Theorem~\ref{Theorem:Binary}, Theorem~\ref{Theorem:Bounded} and the derivation in \cite{Moran} (see the proof of Proposition 5.3) to  give an optimal solution of the subadditive dual problem of~\eqref{eq:generic} that is, in fact, the value function of particular conic MIP. We rigorously state this result in the following corollary.

\begin{cor}
If the conic MIP~\eqref{eq:generic} and its subadditive dual are both feasible, then there exists an optimal dual function $f^*:\rr^m\rightarrow\rr$  that is the value function of a particular conic MIP. More precisely,

\begin{equation*} \label{eq:perturbed3}
\begin{aligned} 
f^*(u) = \inf   &\hspace{0.5em}  c^Tx + d^Ty + [\vartheta^*_{MIP}(b)-\vartheta^*_{MIP}(b-v)] w + [\vartheta^*_{MIP}(b) - 2\Theta^*] s \\
\mathrm{s.t.}   
& \hspace{0.5em} \begin{bmatrix} A & G & v  \\ 0 & 0 & 1  \\ 0 & 0 & -1  \end{bmatrix}  \begin{bmatrix} x \\ y \\  w \end{bmatrix}
-  \begin{bmatrix}  b \\ -\epsilon \\ -(1+\epsilon) \end{bmatrix}  s \succeq_{K \times \mathbb{R}_+  \times \mathbb{R}_+} \begin{bmatrix} u\\ 0\\0\end{bmatrix}  \\ 
& \hspace{.5em}  x  \in   \mathbb{Z}^{n_1}  ,  y \in   \mathbb{R}^{n_2}, w \in   \mathbb{Z}, s \in   \mathbb{Z},
\end{aligned}
\end{equation*}
where  $\epsilon \in (0,1)$,  $v \in \int(K)$ and 
\begin{equation*} \label{eq:perturbed2}
\begin{aligned} 
\Theta^* = \inf  &\hspace{0.5em}  c^Tx + d^Ty + [\vartheta^*_{MIP}(b)-\vartheta^*_{MIP}(b-v)] w\\
  \mathrm{s.t.}   &\hspace{0.5em} A  x + G y + v w\succeq_K b  \\
  & \hspace{0.5em}   w \ge  {-\epsilon}   \\
  & \hspace{-0.56em}   -w  \ge  {-(1+\epsilon)}   \\
  & \hspace{.5em}  x  \in   \mathbb{R}^{n_1}  ,  y \in   \mathbb{R}^{n_2}, w \in   \mathbb{R}.
\end{aligned}
\end{equation*}
\end{cor}

\appendix
\section{Omitted proofs}
\label{app:proofs}

{\black Here, we present the proofs of Lemmas~\ref{dlemma} and~\ref{Fintegerhull} from Section~\ref{sec:DCS}.}

\begin{proof}[Proof of Lemma~\ref{dlemma}]
We must show that there exists a point $w \in \MM$ at a distance less than $\epsilon$ from the half-line $\{z + \lambda r\tq \lambda \geq \gamma\}$. We divide the proof into two cases.

\begin{itemize}
	\item {\bf Case 1: $\MM$ is a lattice.} In this case, there exists $A$, a $n\times k$ matrix with linearly independent columns such that
	$$\MM=\{Az\tq z\in \zz^k\}.$$
	
Since $A$ has  linearly independent columns, $z\in\MM$  and $r\in\aff(\MM)$, there exist $z'\in\zz^k$ such that $z=Az'$ and $r'\in \rr^k$ such that $r=Ar'$. 
	
 By applying the result by {\black Basu et al.~\cite{BCCZ2010}}, for the case $\MM=\zz^k$, we obtain that there exists a point $w' \in \MM$ at a distance less than $\epsilon/\|A\|_2$ from the half-line $\{z' + \lambda r'\tq \lambda \geq \gamma\}$. In particular, there exist $l'$ in this half-line such that $\|w'-l'\|_2\leq \epsilon/\|A\|_2 $. This implies that $\|Aw'-Al'\|_2\leq \|A\|_2 \|w'-l'\|_2 \leq \epsilon$. As $w:=Aw'\in \MM$ and $Al'$ belongs to the half-line  $\{z + \lambda r\tq \lambda \geq \gamma\}$, we conclude the proof.

	\item {\bf Case 2: $\MM$ is a general mixed-lattice.}
	
  Let $\LL\subseteq \MM$ be any lattice such that $z\in \LL$.  By {\bf Case 1}, there exist $ w\in \LL$ at a distance less than $\epsilon$ from the half-line $\{z + \lambda r\tq \lambda \geq \gamma\}$. As $w\in \MM$, we are done.
\end{itemize}
\end{proof}

\begin{proof}[Proof of Lemma~\ref{Fintegerhull}]
	Clearly, $\conv(F\cap\MM)\subseteq\conv(K\cap\MM)\cap F$. We will prove the other inclusion. Let $x\in \conv(K\cap\MM)\cap F$, we will show that $x\in \conv(F\cap\MM)$. {\black We have that} $x=\lambda_1z_1+\ldots + \lambda_pz_p$ for some $z_1,\ldots,z_p\in K\cap \MM$ and $0<\lambda_1,\ldots,\lambda_p\leq 1$ {\black such that $\lambda_1+\ldots+\lambda_p=1$.} 
	Let $D=\conv(\{z_1,\ldots,z_p\})$ and observe that $x$ belongs to the relative interior of $D$. It follows that for any $y\in D$, there exists $y'\in D$ such that $x=\lambda y + (1-\lambda)y'$. Since $y,y'\in K$, $x\in F$ and $F$ is a face of $K$, we obtain that $y,y'\in F$. As $y\in D$ was arbitrary, we obtain that $D\subseteq F$ and, in particular, $z_1,\ldots,z_p\in F$. Therefore, we conclude that $x\in \conv(F\cap\MM)$, as desired.  
\end{proof}

\bibliographystyle{plain}
\bibliography{references}

\end{document}